\documentclass[10pt]{amsart}
\usepackage{amssymb}
\usepackage{float}
\usepackage{graphicx}
\theoremstyle{plain}
\newtheorem{thm}{Theorem}
\newtheorem{prop}[thm]{Proposition}
\newtheorem{lem}[thm]{Lemma}

\newtheorem{cor}[thm]{Corollary}

\newcounter{claimcounter}[thm]
\theoremstyle{definition}
\newtheorem{df}[thm]{Definition}

\newtheorem{exm}[thm]{Example}

\renewcommand{\leq}{\leqslant}

\newcommand{\<}{\langle}

\renewcommand{\>}{\rangle}

\newcommand{\rto}{\rightarrow}

\DeclareMathOperator{\Co}{Co}

\DeclareMathOperator{\CHull}{CHull}

\usepackage{mathtools}

\begin{document}

\title[Closure Operators in geometries with cdim=2]{Description of closure operators in convex geometries of segments on a line}

\author {K. Adaricheva}
\address{Department of Mathematics, Hofstra University, Hempstead, New York, NY 11549, USA}
\email{kira.adaricheva@hofstra.edu}
\author{G. Gjonbalaj}
\address{Hofstra University, Hempstead, New York, NY 11549, USA}
\email{ggjonbalaj1@pride.hofstra.edu}


\begin{abstract}
Convex geometry is a closure space $(G,\phi)$ with the anti-exchange property. A classical result of Edelman and Jamison (1985) claims that every finite convex geometry is a join of several linear sub-geometries, and the smallest number of such sub-geometries necessary for representation is called the convex dimension. In our work we find necessary and sufficient conditions on a closure operator $\phi$ of convex geometry $(G,\phi)$ so that its convex dimension equals 2, equivalently, they are represented by segments on a line. These conditions can be checked in polynomial time in two parameters: the size of the base set $|G|$ and the size of the implicational basis of $(G,\phi)$.
\end{abstract}
% It is wrong!
\maketitle
\section{Introduction}
Convex geometries were studied from different perspectives and under different names since the 1930s. R.P.~Dilworth \cite{D2} knew them as lattices with unique irredundant decompositions, and B.~Monjardet mentions many ways the convex geometries were rediscovered before the mid-80s \cite{Mo85}. An important survey by P.H.~Edelman and R.E.~Jamison \cite{EdJa} included results on several equivalent definitions of finite convex geometries and outlined a program for future studies, suggesting a list of open problems. One of the basic results proved in the paper shows that every finite convex geometry can be generated by a few \emph{linear} sub-geometries on the same base set. The minimal number of such sub-geometries generating the given convex geometry is called the \emph{convex dimension}. This parameter will be referred in the paper as \emph{cdim}. 

Convex geometries are interesting combinatorial objects which generalize a notion of convexity in Euclidean space. There are many structures which share their properties. Among them are convex objects in Euclidean space, convex sets in posets, subsemilattices in a semilattice, and others that are considered in \cite{EdJa}.  But the main driving example is geometrical one: a set of points in Euclidean space equipped with a closure operator of the convex hull. This convex geometry nowadays is called an \emph{affine} convex geometry. 

A new idea was introduced by G.~Cz\'edli \cite{C14}, who suggested using circles rather than points in $\mathbb{R}^2$ for the representation of convex geometries of the convex dimension $2$. It followed from the proof though that to represent geometries of convex dimension 2 one could use the ``circles" on lines rather than planes, i.e., segments on a line.

It was further re-enforced in the work of M.~Richter and L.~Rogers in \cite{RR17}, who showed how, more generally, to use \emph{polygons} for the representation of convex geometries on the plane. In a sense, this is a nice visualization of Theorem 5.2 in P.~Edelman and R.~Jamison \cite{EdJa} about compatible orderings of a convex geometry and its parameter $cdim$. In particular, for the convex geometries of convex dimension 2, one can use segments on a line.

In his paper, G.~Cz\'edli \cite{C14} mentions Carath\'eodory condition $(C_2)$ which is essentially a property of a closure operator of convex geometry $(G,\phi)$: if $a \in \phi(X)$, for some $X\subseteq G, a\in G$, then $a\in \phi(x_1,x_2)$, for some $x_1,x_2 \in X$. It is easy to check that the convex geometries of segments satisfy $(C_2)$, but there was no attempt to check whether the converse is true: will convex geometry with $(C_2)$ have $cdim=2$?

In this paper, we answer to this question in negative, showing that one needs additional properties to get $cdim=2$ in a convex geometry. In fact, one needs a stronger version of $(C_2)$ which we denote (2Ex), and another one which we call the Square Property, or (Sq). 

The main result is Theorem \ref{main}: a convex geometry $(X,\phi)$ has $cdim=2$ iff it satisfies (2Ex) and (Sq). But we also discuss uniqueness of representation of geometries with $cdim=2$ given in Theorem \ref{unique-rep}, which presents an interest of its own.

The paper is organized as follows: in section \ref{DF} we introduce the main concepts and definitions.
%, and provide several statements in the background of our investigation. 
In Section \ref{EP} we discuss the extreme points in general closure spaces, and in section \ref{UR} we describe the convex geometries of $cdim=2$ that have unique representations.
 In section \ref{2Car} we discuss connection between Carath\'eodory property and (2Ex) that is needed for the main result. Finally, section \ref{MT} provides the proof to the main result and discusses the computational complexity of deciding that geometry has $cdim=2$.

\section{Definitions}\label{DF}
     
\begin{df}
Given any set $X$, a closure operator on $X$ is a mapping $\varphi: 2^{X}\rightarrow2^X$ with the following properties:
\begin{enumerate}
\item $Y\subseteq\varphi (Y)$ for every $Y\subseteq X$;
\item If $Y\subseteq Z$, then $\varphi(Y)\subseteq \varphi (Z)$ for $Y,Z \subseteq X$;
\item $\varphi(\varphi(Y))=\varphi(Y)$ for $Y\subseteq X$.
   \end{enumerate}
   \end{df}
 
Set $X$ will be called a \emph{ground set}, or a \emph{base set} for a closure system, the latter being defined as a pair $(X, \varphi)$. We can also associate the closure system with a special family of subsets called an alignment.

\begin{df}\label{Alignment}
Given any (finite) set $X$, an \textit{alignment} on $X$ is a family $\mathcal{F}$ of subsets of $X$ which satisfies two properties:
\begin{enumerate}
\item $X\in \mathcal{F}$;
\item If $Y, Z \in \mathcal{F}$, then $Y \cap Z \in \mathcal{F}$.
  \end{enumerate}
  \end{df}
 Note that the definition of alignment requires a slight modification, when the ground set is not assumed to be finite. We will consider only finite ground sets within the current paper.
 
It is well-known that any alignment on a finite ground set forms a lattice, where the meet operation $\wedge$ is the set intersection $\cap$, and the join operation $+$ is defined as follows: $Y + Z = \bigcap \{W \in \mathcal{F}: Y, Z \subseteq W\}$.

The following relationships between a closure operator and an alignment could be easily verified.
  
  \begin{prop}\label{COvsA}
Let $X$ be some finite ground set.
\begin{enumerate}
\item If $\varphi$ is a closure operator on $X$, then $\mathcal{F}=\{ Y\subseteq X:\varphi(Y)=Y\}$ is an alignment on $X$. 
\item Let $\mathcal{F}$ be an alignment on $X$. Define $\varphi (Y)=\cap\{Z\in \mathcal{F}: Y\subseteq Z\}$ for every $Y\subseteq X$. Then, $\varphi$ is a closure operator on $X$.
\item The correspondences between a closure operator and an alignment on $X$ in items (1) and (2) are inverses of each other. 
\end{enumerate}
 \end{prop}

\emph{An implication} $Y\rightarrow Z$ of a closure system $(X,\varphi)$ is a statement that $Z\subseteq \phi(Y)$. A set $\Sigma$ of such implications is called \emph{an implicational basis}, if any implication that holds in 
$(X,\varphi)$ is a logical consequence of the basis. One usually needs less than $2^X$ implications in order to represent a given closure operator by implications. Thus, $\Sigma$ is thought as a partial information on $\varphi$ which is needed to reconstruct it.

%Often closure systems are defined by set of implications $\Sigma = \{X\rto d: X\subseteq G, X\not = \emptyset, d\in G\}$. The meaning of $X\rto d$ is that $d$ is in the closure of $X$, and 

On the other hand, when the set of implications $\Sigma$ (on finite $X$) is given, one can define a closure operator on $X$ generated by $\Sigma$. It is convenient to use $\Sigma$ also as a notation for this operator. For every $Y\in 2^X$ we define its $\Sigma$-closure as follows. Start with $Y_0=Y$, and define $Y_{n+1}=Y_n\cup \{d: (Z\rto d)\in \Sigma, Z\subseteq Y_n\}$. Then for some $k$ we will have $Y_{k+1}=Y_k$, and we define $\Sigma(Y)=Y_k$.

The study of implicational bases of closure systems, and convex geometries in particular is quite active, see K. Adaricheva and J.B. Nation \cite{AN} and M. Wild \cite{Wi17}.

We turn now to special properties of a closure operator or alignment, which distinguish convex geometries.
 
 \begin{df}
Closure system $(X,\varphi)$ is called a \emph{convex geometry} if $\varphi$ is a closure operator on $X$ with additional properties:
\begin{enumerate}
\item $\varphi(\emptyset)=\emptyset$;
\item if $Y=\varphi(Y)$ and $x, z\notin Y$, then $z\in\varphi(Y\cup x)$ implies that $x\notin\varphi(Y\cup z)$ (The anti-exchange property).
\end{enumerate}
\end{df}

Convex geometries could be defined equivalently through an alignment.

    \begin{df}
 Pair $(X,\mathcal{F})$ is called a \textit{convex geometry} if $\mathcal{F}$ is an alignment on $X$ with additional properties:
 \begin{enumerate}
\item $\emptyset\in \mathcal{F}$;
\item if $Y\in \mathcal{F}$ and $Y \neq X$, then $\exists a\in X\setminus Y$ s.t. $Y\cup\{a\}\in \mathcal{F}$. 
  \end{enumerate}
  \end{df}

Theorem 2.1 in \cite{EdJa} establishes, among other statements, that two above definitions of convex geometry are equivalent.

We will be using both types of presentation of a convex geometry: via a closure operator or an alignment, keeping in mind that one can transfer from one to other seamlessly using Proposition \ref{COvsA}.

There is a simple type of convex geometry whose elements in alignment form a chain.

\begin{df}\label{mon}
Convex geometry $(X,\mathcal{F})$ defined on a set $X$ is called a \emph{linear} if there is a total ordering $x_1<x_2<\dots<x_n$ on $X$ such that $\{x_1, x_2, \dots, x_i\}\in\mathcal{F}$ for all $i, 1\leq i \leq n$, and these sets are the only elements of $\mathcal{F}$.
\end{df}
% It is straightforward to verify that pair  $(X,\mathcal{F})$ is always a convex geometry, for any monotone alignment $\mathcal{F}$ on $X$.

Given two alignments $\mathcal{F}_1$ and $\mathcal{F}_2$ defined on the same base set $X$, an operation of join on these alignments is defined as follows 
:
\[
\mathcal{F}_1 + \mathcal{F}_2=\{S\subseteq X: S=U\cap V \text{ for } U\in \mathcal{F}_1 \text{ and } V\in \mathcal{F}_2\}
\]

The following result was proved as Theorem 5.1 in \cite{EdJa}.

\begin{thm}
If $(X,\mathcal{F}_1)$ and $(X,\mathcal{F}_2)$ are convex geometries, then $(X,\mathcal{F}_1 + \mathcal{F}_2)$ is a convex geometry.  
\end{thm}

For any two geometries on the same ground set $X$, we call $(X,\mathcal{F})$ a \emph{sub-geometry} of  $(X, \mathcal{G})$, if $\mathcal{F}\subseteq \mathcal{G}$. In particular, each of $\mathcal{F}_1$, $\mathcal{F}_2$ is a sub-geometry of $(X,\mathcal{F}_1 + \mathcal{F}_2)$.

It turns out that any convex geometry could be viewed as a join of several linear sub-geometries \cite[Theorem 5.2]{EdJa}:

\begin{thm}\label{root}
Given any convex geometry $G=(X,\mathcal{F})$, $\mathcal{F}=\Sigma_{i \leq n} \mathcal{L}_i$, for some $n \in \mathbb{N}$, where $(X,\mathcal{L}_i)$ is a linear sub-geometry defined on $X$, for every $i\leq n$.
\end{thm}

As a consequence, it is of interest to define a following parameter associated with a convex geometry:

\begin{df}\cite{EdJa}\label{Mon}
Given convex geometry $G=(X,\mathcal{F})$, the convex dimension $cdim$ of $G$  is a minimal number of linear sub-geometries $(X,\mathcal{L}_i), i\leq n$, needed to realize $\mathcal{F}$ as $\Sigma_{i \leq n} \mathcal{L}_i$.
\end{df}

The following example remains to be the main driving model of convex geometries.

\begin{df}\label{affine}
\textit{An affine convex geometry} is a convex geometry $Co(\mathbb{R}^n,X)=(X, ch)$, where $X$ is a set of points in $\mathbb{R}^n$ and $ch$ is closure operator of relative convex hull, which is defined as follows:
for $Y\subseteq X, ch(Y)=\CHull(Y)\cap X$, where $\CHull$ is a usual convex hull operator. 
\end{df}

Affine convex geometries form a sub-class of atomistic closure systems, i.e. systems where each singleton is closed. In particular, not every convex geometry can be realized as affine. To amend this, the following generalization from points to circles was suggested in \cite{C14}. It can easily be generalized to the balls in $\mathbb{R}^n$. 

%\begin{df}
%A closure system $G=(X, \varphi)$ is called \textit{atomistic} if $\varphi(\{x\})=\{x\}$ %for every $x\in X$. 
%\end{df}

%Indeed, for any affine convex geometry $G=(X, ch)$, any $x\in X$ is a point in %$\mathbb{R}^n$ and $ch(\{x\})=\{x\}$. 

%, but we do not need this generalization in the current paper. Note that points are also %considered as circles, whose radii are zero.

\begin{df}\label{GeomSegm}
Consider closure system $F=(X, ch_{c})$, where $X$ is a set of segments in $\mathbb{R}$, or circles in $\mathbb{R}^2$, and closure operator $ch_{c}$ is defined as follows:  $ch_{c}(Y)=\{z\in X: \tilde{z}\subseteq \CHull(\cup \tilde{y}: y\in Y)\}$ for $Y\subseteq X$, where $\CHull$ is a usual convex hull operator and $\tilde{x}$ is a set of points in $x \in X$. We call $F$ a \textit{geometry of segments on a line}, or \textit{geometry of circles on a plane}. 
\end{df}

The following statement is an easy consequence of \cite[Lemma 2.4]{Kinc17}.

\begin{prop} For any finite set of segments $X$ on a line that do not have common end-points, the closure system $F=(X, ch_{c})$ is a convex geometry.
\end{prop}

A geometry of circles is not atomistic in general. For convex geometry of circles $F=(X, ch_c)$, it is possible that $ch_c(\{x\})=\{x,y\}$ for $x, y \in X$ that describes a case when circle $y$ is inside circle $x$. 

%Important concept for the current paper relates two geometries through the mapping of %embedding.
    
 %    \begin{df}\label{Subgeometry}
%Convex geometry $G_1=(X,\mathcal{F}_1)$ is a sub-geometry of $G_2=(Y, \mathcal{F}_2)$ if %there is a one-to-one map $f:\mathcal{F}_1\rightarrow \mathcal{F}_2$ s.t. 
%\begin{enumerate}
%\item $f(A\cap B)=f(A)\cap f(B)$, $A, B \subseteq X$;
%\item $f(A\vee B)=f(A)\vee f(B)$, $A, B \subseteq X$.
% \end{enumerate}
% \end{df}

%Another way to connect these two geometries is to say that geometry $G_1$ has a \emph{weak %representation} in $G_2$.

%If map $f$ is also onto, then we talk about \emph{strong representation} of $G_1$ in %$G_2$, or isomorphism of convex geometries. Such mapping also induces bijection $f_g$ %between ground sets $X$ and $Y$, namely, $f_g(x)=y$ iff $f(\phi_1(\{x\})) = %\phi_2(\{y\})$, where $\phi_1$ and $\phi_2$ are closure operators corresponding to %alignments $\mathcal{F}_1$ and $\mathcal{F}_1$, respectively.

\section{Extreme points in closure systems}\label{EP}

We start by recalling that for every closure system $(G, \phi)$ and for every $S\subseteq G$ one can define a \emph{restriction} of $\phi$ on $S$ as follows: $\phi_S(Y)=\phi(Y)\cap S$, for every $Y\in 2^S$. See, for example Lemma 2-2.2 in \cite{AN}, which states that $\< S, \phi_S\>$ is a closure system on $S$. Moreover, if $(G,\phi)$ is a convex geometry then $(S, \phi_S)$ is also a convex geometry \cite{EdJa}.

Recall that $x \in G$ is called an \emph{extreme point} of $(G, \phi)$, if $x \not \in \phi(G\setminus \{x\})$. The set of all extreme points of $(G, \phi)$ is denoted $Ex(G)$.

We will be interested in a restriction of convex geometry $(G, \phi)$ on $S\subseteq G$ such that $G\setminus S\subseteq Ex(G)$.

We want to establish the implicational version of the restriction of closure system $(G, \phi)$ on $S\subseteq G$. If $\Sigma$ is the set of implications on $G$ and $S\subseteq G$, then we denote $\Sigma_S = \{X\rto d: X\subseteq S\}\subseteq \Sigma$.

\begin{prop}
Let $(G, \phi)$ be a closure system on finite set $G$ and $S\subseteq G$ such that $G\setminus S\subseteq Ex(G)$. If $\Sigma$ is any implicational basis of $(G, \phi)$, then $\Sigma_S$ is an implicational basis of $(S, \phi_S)$.
\end{prop}

\begin{proof}
We need to show that $\phi_S(Y)=\Sigma_S(Y)$, for every $Y\in 2^S$.
Since $G\setminus S$ holds only extreme points, we have $\phi(Y)\subseteq \phi(S)\subseteq S$, thus, $\phi_S(Y)=\phi(Y)$. 

We claim that also $\Sigma(Y)=\Sigma_S(Y)$. On one hand, apparently, 
$\Sigma_S(Y)\subseteq \Sigma(Y)$, because $\Sigma_S\subseteq \Sigma$. 

The inverse inclusion follows from the observation that every $Y_{n+1}$ in computation with respect to $\Sigma$ will only include elements $d$ from implications $X\rto d$, where $X\subseteq Y_n$, and we can show by induction that $Y_n\subseteq S$, therefore only implications from $\Sigma_S$ are actually used to build $Y_{n+1}$. Note that any implication $X\rto d$ in $\Sigma_S$ cannot have $d \in G\setminus S$. This is why every $Y_n\subseteq S$, inductively from assumption $Y\subseteq S$. 

Thus, $\phi_S(Y)=\phi(Y)=\Sigma(Y)=\Sigma_S(Y)$, and we are done.
\end{proof}

We will also need an easy observation about extreme points.

\begin{prop}\label{ex}
If $x \in S$, $S\subseteq G$ and $x$ is an extreme point of $(G, \phi)$, then it is an extreme point of restriction $(S, \phi_S)$ of $(G, \phi)$. 
\end{prop}
\begin{proof}
Indeed, if $x \not \in \phi(G\setminus \{x\})$, then $x \not \in \phi(S\setminus \{x\})$.
\end{proof}

We end this section by important characterization of convex geometries via extreme points.

\begin{thm}\label{extp}\cite{EdJa}
A closure system $(X,\phi)$ is a convex geometry iff $Y=\phi(Ex(Y))$, for any closed set $Y\subseteq X$.
\end{thm}

\section{Geometries of convex dimension 2}\label{UR}

It was proved in \cite{C14} that every convex geometry of $cdim=2$ is represented as $(X, ch_c)$, for some set $X$ of circles on the plane, but it was also noted that, in effect, just segments on a line are needed for representation. This observation was also reinforced in \cite{RR17}, where the representation by segments was realized in the spirit of Theorem \ref{root}. We give an example of representation for $cdim=2$ used in \cite{RR17} in the following example, and then summarize these results in Theorem \ref{segments} for completeness of the argument.

Consider convex geometry $(X,\phi)$ of $cdim=2$. By Definition \ref{Mon}, there exist two strict linear orders $<_L$ and $<_R$ on set $X$ such that $(X,\phi)=(X,<_L)+ (X, <_R)$, where $(X, <_L)$ is a notation of geometry on $X$ with monotone alignment defined by $<_L$, and similarly for $(X, <_R)$. 

\begin{exm}\label{un}
\end{exm}
Let $X=\{a,b,c,d\}$. Consider two linear orders: $a<_L b<_L c <_L d$ and $c<_R b<_R d <_R a$ and let $(X,\phi)=(X,<_L)+ (X, <_R)$.
We check that the alignment of $(X,<_L)$ is $\mathcal{F}_L=\{\emptyset, \{a\},\{a,b\},\{a,b,c\},X\}$ and alignment of $(X, <_R)$ is $\mathcal{F}_R=\{\emptyset, \{c\},\{c,b\},\{c,b,d\},X\}$. These two linear alignments generate alignment $\mathcal{F}=\mathcal{F}_L+\mathcal{F}_R = \mathcal{F}_L\cup \mathcal{F}_R \cup \{\{b\}\}$, which coincides with the set of all closed sets of closure operator $\phi$, see Proposition \ref{COvsA}. 

This geometry can be turned into geometry of segments on a line: take a copy of $\mathbb{R}$, pick any $r \in \mathbb{R}$, say, $r=0$, then place copies of elements from $X$ in the ray of negative numbers, in the increasing order given by $<_L$, when moving from $0$ to $-\infty$, and place copies of elements from $X$ in the ray of positive numbers, in the increasing order given by $<_R$, moving from $0$ to $+\infty$. Indicating the position of $r=0$ by $\nabla$, we  can formally present it as 
%$d\ c\ b \ a \nabla c\ b\ d\ a$. Note that $\nabla$ is a separator of two copies of $X$, where the left copy is ordered by $_L\geq$ increasing to the left, and the right copy is ordered by $\leq_R$ increasing to the right.  We will use notation 
$(d\ c\ b \ a \nabla c\ b\ d\ a)$. Every element of $X$ can now be identified by a segment on a line with end-points being two copies of that element on the left and right from $\nabla$.
%for representation of geometry $(X,\phi)$ with these two orders. 

For example, check that $\{b\}$ is a $\phi$-closed set: segment with end-points $b$ on both sides of $\nabla$ does not contain any other segment $a,c,d$. On the other hand, $\phi(d)=\{d,c,b\}$, which is manifested by having segments $c,b$ inside segment $d$.
\\

\begin{thm}\label{segments}
Let $(X,\phi)$ be a convex geometry. Then the following statements are equivalent:
\begin{itemize}
\item[(1)] $(X,\phi)$ has $cdim=2$;
\item[(2)] $(X,\phi)$ has a representation by segments on a line.
%\item[(3)] $(X,\phi)$ satisfies properties (2Ex) and (ExR).
\end{itemize}
\end{thm}

\begin{proof}
(1) implies (2) using a representation in \cite{RR17} and \cite[Theorem 5.2]{EdJa}.  

We shall prove
that (2) implies (1). Assume that $(X,\phi)$ is set with a closure operator $\phi$ and 
the anti-exchange property. Assume that $(X,\phi)$ has a representation of line 
segments $\{[a_i,b_i] : 1\le i\le n\}$, where $a_i,b_i\in \mathbb{R}$. Our goal 
is to show that this representation is isomorphic to the construction of \cite{RR17}, and implying that $(X,\phi)$ is $cdim=2$. 

On that note, let 
$\alpha=$max$\{a_1,a_2,...,a_n\}$ and $\beta=$min$\{b_1,b_2,...,b_n\}$. Observe that if
$\beta > \alpha$ then our geometry has the same representation as in \cite{RR17}, 
hence, $cdim=2$. So assume that $\beta<\alpha$. Let $c=(\alpha-\beta)+1$. We 
now form a new convex geometry $(X^*,\phi^*)$ where the representation of the geometry 
has line segments of the form $\{[a_i,b_i + c] : 1\le i \le n\}$. We notice that 
$(X,\phi^*)$ is a geometry that has the same representation as in \cite{RR17}, hence, its convex dimension is 2. 

Therefore, to complete the proof we will show that 
$(X,\phi)$ is isomorphic to $(X^*,\phi^*)$. We claim $f: X\rightarrow X^*$ by 
$f([a_i,b_i]=[a_i,b_i+c]$ is an isomorphism. Clearly $f$ is one-to-one and onto map of one ground set to another. It remains to show that $f$ induces the mapping of closed sets to closed set. So we need to check $x \in \phi(Y)$ iff $f(x) \in \phi^*(f(Y)$, for any $Y\subseteq X$. Moreover, for any set of intervals on the line, $\phi(Y)=\phi(y_1,y_2)$ for some $y_1,y_2\in Y$.

So we have $[a_j,b_j]\in \phi([a_i,b_i], [a_k,b_k])$ iff $\min(a_i,a_j) < a_j< b_j < \max(b_i, b_k)$ iff $\min(a_i,a_j) < a_j< b_j+c < \max(b_i+c, b_k+c)$ iff $[a_j,b_j+c]\in \phi^*([a_i,b_i+c], [a_k,b_k+c])$, which is what is needed.

\end{proof}

Going back to Example \ref{un}, we note that one can switch two orders, placing $<_L$ on the right and $<_R$ on the left:
$(a\ d\ b \ c \nabla a\ b\ c\ d)$. This is the same representation, since both chains did not change.
We want to investigate when representation of geometry with $cdim=2$ is unique up to a switch of two chains, so  
now we give an example when geometry of $cdim=2$ has more than one representation.

\begin{exm}\label{switch}
\end{exm}
Let $X=\{a,b,c,1,2,3\}$, and convex geometry on $X$ of $cdim=2$ is represented as
$(b \ a\ c\ 2 \ 1\ 3 \ \nabla \ 2\ 3\ 1\ c\ b\ a )$. Note that the subset $\{1,2,3\}$ is an initial segment of each of two chains, while $\{a,b,c\}$ represents an ending segment of both chains. In particular, $x\rightarrow p$, for every $x \in \{a,b,c\}$ and $p \in \{1,2,3\}$. We also notice that there is a further splitting within group $\{a,b,c\}=\{c\}\cup \{a,b\}$ so that $a\rightarrow c$ and $b \rightarrow c$, which is reflected in the fact that subchains of two chains on $\{a,b,c\}$ start with $c$ and end with $\{a,b\}$ ordered specifically in each chain.

Apparently, one can switch portions of chains on subset $\{1,2,3\}$ between left and right chains, while leaving the rest of elements ordered as before, to obtain the same convex geometry, but represented by two different chains: $(b \ a\ c\ 1 \ 3\ 2 \ \nabla \ 3\ 1\ 2\ c\ b\ a )$.\\

This example is generalized in the following statement.

\begin{lem}\label{suffi}
Let $(X,\phi)$ be a convex geometry of $cdim=2$, and $X=\{p_1,\dots, p_k\}\cup \{q_1,\dots q_s\}$, where $k,s>1$. If $(X, \phi)=(X,<_L)+ (X, <_R)$ such that both of two chains start with some distinct permutations $i,j:\mathbf{s}\rightarrow \mathbf{s}$ of elements $\{q_1,\dots q_s\}$, and end with some distinct permutations $m,n:\mathbf{k}\rightarrow \mathbf{k}$ of elements $\{p_1,\dots, p_k\}$, i.e. $q_{i(1)}<_L q_{i(2)}<_L\dots<_L q_{i(s)} <_L p_{m(1)}<_L p_{m(2)}<_L\dots<_R p_{m(k)}$ and $q_{j(1)}<_R q_{j(2)}<_R\dots<_R q_{j(s)} <_R p_{n(1)}<_R p_{n(2)}<_R\dots<_R p_{n(k)}$, then there exists another representation of $(X,\phi)$. 
\end{lem}
%\begin{proof}
%Straightforward argument, based on the idea of Example \ref{switch}.
%\end{proof}

\begin{cor}
Under assumptions of the previous Lemma, let $X_p=\{p_1,\dots, p_k\}$ and $X_q=\{q_1,\dots q_s\}$, and let $(X_p,\phi_p)$ and  $(X_q,\phi_q)$ be restrictions of $(X,\phi)$ on subsets $X_p$ and $X_q$, respectively. Then both restrictions have $cdim=2$ and arbitrary representations of $(X,\phi)$ can be obtained by combining one representing chain for $(X_p,\phi_p)$ with one representing chain of $(X_q,\phi_q)$, then second representing chain of $(X_p,\phi_p)$ with second representing chain of $(X_q,\phi_q)$. In particular, if both  $(X_p,\phi_p)$ and $(X_q,\phi_q)$ have unique representations, then $(X,\phi)$ will have two representations. 
\end{cor}

Lemma \ref{suffi} describes the sufficient condition for multiple representations of a convex geometry: both chains in representation have ending segments that are identical as sets. We can also prove that such condition is necessary. First, illustrate the method of Lemma \ref{neces} on the following

\begin{exm}\label{unique}
\end{exm}
Suppose a convex geometry on $X=\{1,2,3,4,5\}$ is defined by two chains whose ending segments are distinct: $(5 \ 1\ 3\ 2\ 4\ \nabla \ 2\  1\ 3\ 5\ 4)$. In fact, one only needs to check this for ending segments of length 1,2 and $3=n-2$, where $n=5=|X|$.

By assumption, the ending elements  of two chains are distinct: $5$ and $4$. They are also extreme points of $(X,\phi)$. Without loss of generality, we could assume that $5$ is $<_L$-maximal and $4$ is $<_R$-maximal. Verify that $Ex(X\setminus\{5\}) = \{1,4\}$, where $4$ remains to be extreme point, therefore, $1$ is the second-to-maximal element in $<_L$. Similarly, $Ex(X\setminus\{4\})=\{5\}$, therefore, $5$ is the second-to-maximal element in $<_R$.

On the next step, we want to remove two largest elements in $<_L$ : 5 and 1. Among the  
two largest elements in $<_R$ there should be at least one distinct from 1 and 5, in our case $4$. Therefore, $4$ will be among $Ex(X\setminus\{1,5\})=\{3,4\}$, and the second extreme point $3$ will be the third-to-maximal element in $<_L$. Similarly, removing 5 and 4, which are the two largest elements in $<_R$ will reveal the third-to-maximal in $<_R$ element as $3$. Indeed, $Ex(X\setminus\{1,5\})=\{1,3\}$, where $1$ is the element in the segment of length two in $<_L$ order that is distinct from 4 and 5. This identifies $3$ as the third-to-maximal element in $<_R$. 

On the last step, we find $Ex(X\setminus\{5,1,3\})=\{2,4\}$, which gives the forth-to-maximal element in $<_L$ as $2$, and $Ex(X\setminus\{3,5,4\})=\{1\}$, which gives the forth-to-maximal element in $<_R$ as $1$. Since 4 elements of both chains are identified, the minimal element in each is uniquely determined as well.

%First, introduce a property that will also be essential for our main result in the next section.

\begin{lem}\label{neces}
Suppose  $(X, \phi)=(X,<_L)+ (X, <_R)$ is a representation of convex geometry of $cdim=2$ and $|X|=n$. If, for every $1\leq k \leq n-2$ the ending segments of $k$ elements of chains $(X,<_L)$ and $(X,<_R)$ are distinct as sets, then $(X,\phi)$ has a unique representation.
\end{lem}
\begin{proof}
We will show that, under the assumption of Lemma, the orders of both representing chains are uniquely determined by the knowledge of extreme points of various subsets of $X$, and the latter is independent of any representation.

Let $Ex(X)=\{x_1,x_2\}$, and $x_1\not = x_2$ by assumption of Lemma.
%In the first case, both $(X,\leq_L)$ and $(X, \leq_R)$ have $x_1$ as a maximal element. Otherwise, 
Without loss of generality, we can assume that $x_1$ is the $<_L$-maximal element and $x_2$ is $<_R$-maximal element. 

Now consider $X_1=X\setminus x_1$. Then we have  $Ex(X_1)=\{x_2,x_3\}$, where $x_2,x_3$ may or may not be equal. Indeed, if $x_2 \in Ex(X)$, then it remains to be in $Ex(X_1)$, by Corollary \ref{ex}. 
%If $|Ex(X)|=1$, then, by assumption, $x_2\not = x_3$, and,  we can choose, without loss of generality, that $x_3$ is the $\leq_L$-maximal element and $x_2$ is $\leq_R$-maximal element. 
We observe that $x_3$ must be the second largest element in $<_L$ order. 
%In other words, we know that, in any representation, when $Ex(X)|=2$ we would have the following endings of two chains: $(x_1 \  x_3 \nabla x_2)$. We want to show that now second largest element in $\leq_R$ is also uniquely defined. Suppose in existing representation we have, for two largest elements of two chains: $(x_1 \  x_3 \nabla x_4 \ x_2)$. 
Similarly, wet take $X_2=X\setminus x_2$. Then $Ex(X_2)=\{x_1,x_4\}$, where $x_1,x_4$ may or may not be distinct. This determines $x_4$ as the second largest element in the right chain, for any representation.

Thus, we may assume that we were able to determine uniquely the two largest elements of both chains, up to the switch of two chains.

Now assume that we were able to determine $k$ largest elements of both chains:\\$(x_1 \  x_3 \dots \ x_{2k-1}\nabla x_{2k}\ \dots x_4 \ x_2)$, for some $1\leq k \leq \lfloor (n-2)/2\rfloor $. By assumption, one of elements, say, $x_{2s-1}$ on the left is distinct from those on the right, and we may assume that it is the largest among those that do not appear on the right. Take $X_{2k+2}=X\setminus \{x_2,x_4,\dots x_{2k}\}$. Then $Ex(X_{2k+2})=\{x_{2s-1}, x_{2k+2}\}$, where two points may or may not be equal. This determines that $x_{2k+2}$ is the largest element of $X_{2k+2}$ on the right. Similarly, we could determine what is the largest element on the left, considering $Ex(X\setminus\{x_1,\dots x_{2k-1}\})$.

We could proceed with this argument up to $k=\lfloor (n-2)/2\rfloor$, thus, determining, the second points on the left and on the right. Therefore, the first points will be determined uniquely also.
%By assumption, $\{x_1,x_3\}\not = \{x_2,x_4\}$. 
\end{proof}

The following statement will play the role in the proof of the main result of the next section. Note that we use a slightly stronger assumption on convex geometry than in the previous Lemma. In particular, every geometry in the Lemma below will have a unique representation.

\begin{lem}\label{seq}
Suppose  $(X, \phi)=(X,<_L)+ (X, <_R)$ is a convex geometry of $cdim=2$ with exactly two extreme points, and $|X|=n$. If, for every $1\leq k \leq n-1$ the ending segments of $k$ elements of chain $(X,<_L)$ and $(X,<_R)$ are distinct as sets, then elements of $X$ can be ordered $x_1,x_2,\dots, x_n$ in such a way that the following holds:
\begin{itemize}
    \item [(1)] $x_1$ is an extreme point of $(X, \phi)$;
    \item[(2)] $x_i$ is an extreme point of $(X\setminus\{x_1,\dots, x_{i-1}\}, \phi)$, for all $i=2,\dots, n-1$;
    \item [(3)] $(X\setminus\{x_1,\dots, x_{i-1}\},\phi)$ has exactly two extreme points, for all $i=2,\dots, n-1$.
\end{itemize}
\end{lem}
\begin{proof}
Let us call (ES) the property about distinct ending segments of two chains in the formulation of Lemma. Suppose $a\not= b$ are two extreme points of $(X,\phi)$, and without loss of generality we assume that $a$ is $<_L$-maximal element and $b$ is $<_R$-maximal. 

We claim  that either removing $a$ or removing $b$ one obtains two chains:\\ $(X\setminus \{t\} , <_L)$, $(X\setminus \{t\}, <_R)$, $t\in \{a,b\}$, with property (ES). If we prove the claim, then by induction we can find a sequence of removed extreme points satisfying  statements (1),(2),(3).

Suppose this is not the case, and  $(X\setminus \{t\}, <_L)$ and  $(X\setminus \{t\}, <_R)$ have equal ending segments, for $t=a$ and for $t=b$. Then chains of $X\setminus \{a\}$ will look like $(S_b^* S_{\neg b}^*\nabla S_{\neg b}S_b)$, where $S^*_b, S_b$ are ending sectors of chains $(X\setminus \{a\},<_L), (X\setminus \{a\}, <_R)$, and these sectors are equal as sets. Subscript $b$ indicates that $b \in S_b^*,S_b$. Respectively, $b$ does not appear in sectors $S_{\neg b}^*, S_{\neg b}$. 

Similarly, the chains of $X\setminus \{b\}$ will look like $(S_a^* S_{\neg a}^*\nabla S_{\neg a}S_a)$, where $S^*_a, S_a$ are ending sectors of chains $(X\setminus \{b\},<_L), (X\setminus \{b\}, <_R)$, and these sectors are equal as sets. Without loss of generality we may assume that $|S_{\neg b}|\leq |S_{\neg a}|$, which means that $S_{\neg b}$ is a subsequence in $S_{\neg a}$. In particular, $S_{\neg b}$ does not have both $a$ and $b$, and, similarly, $S_{\neg b}^*$, which is a permutation of $S_{\neg b}$.  This will imply that $S_{\neg b}^*, S_{\neg b}$ are initial segments of original chains $(X,<_L), (X, <_R)$ and they coincide as sets, leading to the fact that the ending segments of two chains equal as sets as well, which contradicts to the assumption of Lemma.    
\end{proof}

Combining Lemmas \ref{suffi} and \ref{neces}, we can formulate the statement about unique representation.

\begin{thm}\label{unique-rep}
Convex geometry $(X,\phi)=(X,<_L)+ (X, <_R)$ of $cdim=2$ has a unique representation iff there exists $X'\subseteq X$ such that the restriction $(X', \phi')$ satisfies property of Lemma \ref{neces} and elements of $X'$ fill the same segments of two representing chains, while restriction on $X\setminus X'$ has $cdim=1$. 
\end{thm}

\begin{exm}
\end{exm} To illustrate the statement of Theorem \ref{unique-rep}, consider
$X=\{a,b,c,d,1,2,3\}$, and the geometry on this set which given by segments on the line:
\[
(3 \ d \ b \ a\ c\ 2 \ 1\  \nabla \ 1\ 2\ d \ c\ b\ a\ 3). 
\]
Observe that the restriction on $X'=\{a,b,c,d\}$ has the property of Lemma \ref{neces}: the 2-element segments at the max-end of each chain are distinct subsets of $X'$, and all elements of $X'$ populate the same segment from third to sixth element of each chain. The projection on $X\setminus X'$ gives two identical sub-chains, thus, the restriction $(X\setminus X', \phi)$ has $cdim=1$. \\

We can briefly discuss how the general convex geometry on set $X$ of $cdim=2$ can be envisioned, with respect to its (possibly, multiple) representations. There is a partition $X=P_1\cup P_2\dots \cup P_n$ of base set, and segment representation scheme, where subsets of elements are placed rather than individual elements:\\ $(P_n\  \dots \ P_2 \ P_1\  \nabla \ P_1\ P_2\ \dots P_n)$. Projection of geometry on each $P_i$ has a unique representation, and multiple representations can be obtained by switching left and right chains of individual segments $P_i$.

We note, in particular, that $u\rightarrow w$, for every $u \in P_k$ and $w \in P_s$, where $k>s$. 

\section{2-Carath\'eodory property and the others}\label{2Car}

In this section we discuss the 2-Carath\'eodory property and its stronger and weaker versions.

\begin{df}\label{Car1}
For a closure space $(X,\phi)$ the $n$-Carath\'eodory property holds, if $a \in \phi(X')$, for some $X'\subseteq X, a\in X$, implies that $a\in \phi(x_1,\dots x_{n})$, for some $x_1,\dots x_{n} \in X'$.
\end{df}

It is well-known that if $X=\mathbb{R}^{n-1}$ and $\phi$ is the convex hull operator, then the space $(X,\phi)$ satisfies the $n$-Carath\'eodory property. In particular, $2$-Carath\'eodory property holds for a convex hull operator in one-dimensional space.

We will consider two more properties formulated in the same vein. 

\begin{df}\label{2impl}
For a closure space $(X,\phi)$ the ($2$Impl) property holds, if there exists an implication basis $\Sigma \{A_i\rightarrow B_i:i\leq k\}$ for $(X,\phi)$ that satisfies $|A_i|\leq 2$, for all $i\leq k$.
\end{df}

\begin{df}\label{2Ex}
Closure space $(X,\phi)$ satisfies property (2Ex), if, for every $X'\subseteq X$, $|Ex(X')|\leq 2$.  
\end{df}

\begin{prop}\label{2and1}
For any closure space $(X,\phi)$ the following hold:
\[
\text{(2Ex) } \Longrightarrow (2-\text{Carath\'eodory }) \Longrightarrow  (2\text{Impl}).
\]
\end{prop}

\begin{proof}
Suppose $(X,\phi)$ satisfies (2Ex), and let $a \in \phi(X')$, for some $X'\subseteq X, a\in X$. By assumption, $\phi(X')$ has at most two extreme points, say, $Ex(\phi(X'))=\{x_1,x_2\}$, where $x_1,x_2 \in X'$. Due to Theorem \ref{extp}, every closed set of convex geometry is generated by its extreme points. Therefore, $a \in \phi(x_1,x_2)$, and 2-Carath\'eodory holds.

Now suppose 2-Carath\'eodory holds in $(X,\phi)$. Consider any implicational basis $\Sigma$ for $(X,\phi)$. For every implication $X'\rightarrow a$ in $\Sigma$, consider $x_1,x_2 \in X'$ such that $x_1x_2\rightarrow a$, due to Carath\'eodory property. This last implication follows from $\Sigma$. On the other hand, $X'\rightarrow a$ follows from $x_1x_2\rightarrow a$. Therefore, collecting shorter implications into set $\Sigma'$ we obtain a new basis of $(X,\phi)$, thus, (2Impl) holds.
\end{proof}

None of the implications of Proposition \ref{2and1} can be reversed.\\

(1) Closure space may satisfy 2-Carath\'eodory property, but not (2Ex).

\begin{exm}
\end{exm}

Consider an example of affine convex geometry $\Co(\mathbb{R}^2,X)$ (see Definition \ref{affine}), where $X=\{a,b,c,x\}$ is a set of points on a plane, points $a,b,c$ are not on a line and $x$ is on the segment $[a,b]$. Apparently, 2-Carath\'eodory property holds. On the other hand, set $X$ has three extreme points $a,b,c$.\\

(2)  Closure space may satisfy (2Impl), but not 2-Carath\'eodory property.

\begin{exm}
\end{exm}
Consider an example of affine convex geometry $\Co(\mathbb{R}^2,X)$, where $X=\{a,b,c,d,x\}$ is a set of points on a plane, points $a,b,c$ are not on a line, $d$ in on segment $[b,c]$ and $x$ is on the segment $[a,d]$.

There exists basis $\Sigma=\{bc\rightarrow d, ad\rightarrow x\}$ for $\Co(\mathbb{R}^2,X)$ that satisfies (2Impl). On the other hand, $x \in \Co(a,b,c)$, but not in $\Co(u,v)$, for any $u,v \in \{a,b,c\}$.\\

The following statement strengthens \cite[Proposition 3.8]{C14}
\begin{lem}
If $(X,\phi)$ is a convex geometry with $cdim=2$, then it satisfies (2Ex).
\end{lem}
\begin{proof}
This follows from Theorem \ref{segments}.
\end{proof}

While (2Ex) is a necessary condition for the convex geometries of $cdim=2$, it is not sufficient.

\begin{exm}\label{notsuf}
\end{exm}

Consider convex geometry $(X,\phi)$ with $X=\{a,b,c,d\}$, defined by the set of implications: $ab \rightarrow c$, $bc\rightarrow d$, $a\rightarrow d$.

\begin{figure}[H]
 \includegraphics[width=0.5\textwidth]{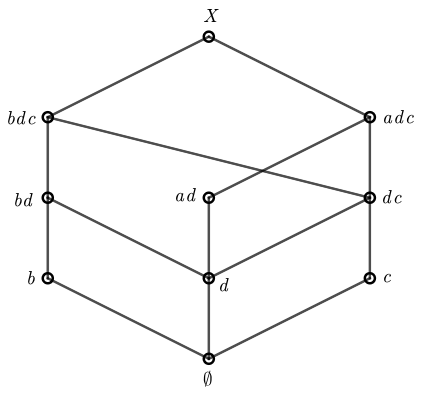}
 %{four.jpg}
 \caption{}
 \label{fig:four}
 \end{figure}

The Hasse diagram of this geometry is shown on Figure \ref{fig:four}. It is straightforward to verify that every closed set has no more than 2 extreme points.
On the other hand, there is no two chains from $\emptyset$ to $X$ that generate all closed sets of this convex geometry using $\cap$ : closed sets $\{b,d\}, \{a,d\}$ and $\{c\}$ are incomparable and meet-irreducible, so they all have to be in distinct chains. 

This example shows that one needs additional property to make characterization of closure operator of convex geometry with $cdim=2$.

\section{Main Result}\label{MT}

In this section we prove our main result, and we start by introducing a new property and discussing its equivalent formulation.

\begin{df}
Closure space $(X,\phi)$ satisfies property (Sq), or the Square Property, if, for any $X'\subseteq X$, $Ex(X')=\{a,b\}, Ex(X'\setminus a)=\{c,b\}$ and $Ex(X'\setminus \{a,b\})=\{c,d\}$, where $a\not = b\not = c$, imply $Ex(X'\setminus b)=\{a,d\}$ or $Ex(X'\setminus b)=\{a\}$.
\end{df}

 \begin{figure}[H]
 \includegraphics[width=0.5\textwidth]{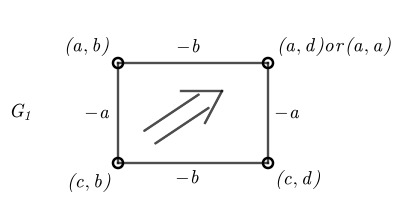}
% {Sq.jpg}
 \caption{Square Property}
 \label{fig:Sq}
 \end{figure}

%\begin{df}
%Closure space $(X,\phi)$ satisfies property (2Ex), if, for every $X'\subseteq X$, %$|Ex(X')|\leq 2$.  
%\end{df}
The property is presented on Figure \ref{fig:Sq}, where the ordered pair $(a,b)$ at the upper left corner includes both (distinct) extreme points of sub-geometry $X'$. Moving horizontally or vertically  gives a restriction of geometry onto a subset, obtained by removing one or two extreme points. Which of points is removed is indicated by marking the edge of a square. Extreme points are shown as ordered pairs indicating that the vertical action of removing $a$ changes the first component of an ordered pair, and horizontal action of removing $b$ changes the second component of the ordered pair.  The assumption of the property is on two edges connecting at pair $(c,b)$: $a\not = b\not = c$, then the conclusion is indicated by a large arrow inside the square on what should be expected at the upper right corner.

Note that the case $(a,a)$ at the upper right corner implies that $a \rightarrow g$, for every $g \in X_1\setminus\{b\}$. In particular, $a\rightarrow c,d$.

\begin{exm}\label{notsuf2}
\end{exm}

Geometry $(X,\phi)$ from Example \ref{notsuf} fails (Sq). Indeed, taking $X'=X$ with same elements as on Figure \ref{fig:Sq}, we have $Ex(X')=\{a,b\}, Ex(X'\setminus a)=\{c,b\}$ and $Ex(X'\setminus \{a,b\})=\{c,d\}$. On the other hand, $X\setminus \{b\}=\{a,d,c\}$, therefore, $Ex(X'\setminus b)=\{a,c\}$, which contradicts the conclusion of (Sq).\\

We now formulate a similar property. 
We will use notation $a\not \rightarrow  z$ for $z \not \in \phi(a)$.

\begin{df}
Closure space $(X,\phi)$ satisfies property (ExR), if for every $X'\subseteq X$ with $Ex(X')=\{a,b\}$ and $Ex(X'\setminus \{a\})=\{c,b\}$, the assumptions [$a \not \rightarrow z$ and ($a\rightarrow y$ or $az\rightarrow y$)] imply [$cz\rightarrow y$], for all $y,z \in X'\setminus \{a\}$ . 
\end{df}

This property is saying that when element $c$ 'replaces' $a$ as an extreme point, after $a$ is removed, $c$ also 'replaces' $a$ in some implications.

Note that the property indirectly assumes that $a\not = b$, because, otherwise, $b \not \in X'\setminus \{a\}$. Moreover, there will be no $z \in X'$ with $a\not \rightarrow z$. Also, if $b=c$, then property always holds, because conclusion $cz\rightarrow y$ would follow from $Ex(X'\setminus \{a\})=\{b\}=\{c\}$. Therefore, one can always assume $a\not = b\not = c$ in (ExR).

\begin{lem}\label{ExR}
Let $(X,\phi)$ be a convex geometry with property (2Ex). Then it satisfies (ExR) iff it satisfies (Sq).
%If $(X,\phi)$ is a convex geometry with $cdim=2$, then it satisfies (ExR).
\end{lem}
\begin{proof}
First we show that if (Sq) fails then (ExR) fails as well. Note that if $c=d$, (Sq) always holds: if $Ex(X_1)=\{a,b\}, Ex(X_1\setminus \{a,b\})=\{c\}$ and $Ex(X_1\setminus \{a\})=\{c, b\}$, then $Ex(X_1\setminus \{b\})=\{c, a\}$.

Therefore, (Sq) fails,  if for some subset $X_1$ the assumptions hold: $Ex(X_1)=\{a,b\}, Ex(X_1\setminus a)=\{c,b\}$ and $Ex(X_1\setminus \{a,b\})=\{c,d\}$, where $a\not = b\not = c\not = d$ - but the conclusion fails, i.e. $Ex(X_1\setminus b) \not=\{a,d\}$ and $Ex(X_1\setminus b)\not=\{a\}$.

Now, we know that $a \in Ex(X_1\setminus b)$ and it is not the only extreme point. By (2Ex) there should be exactly one more extreme point, moreover, it should be one of $Ex(X_1\setminus \{a,b\})=\{c,d\}$. Thus, by assumption, there is only one remaining possibility: $Ex(X_1\setminus b)=\{a,c\}$. In particular, $ac\rightarrow d$ should hold, and $a\not \rightarrow c$, because, otherwise, $c \not \in Ex(X_1\setminus b) =\{a,c\}$, a contradiction. Now, if (ExR) holds, then the assumptions: $Ex(X_1)=\{a,b\}, Ex(X_1\setminus a)=\{c,b\}$, $a\not \rightarrow c$ and $ac\rightarrow d$ - would imply $c\rightarrow d$. This contradicts that $d$ is extreme point of $X_1\setminus \{a,b\}$ different from $c$. Therefore, (ExR) fails.

Secondly, we assume that (Sq) holds and show that (ExR) holds. So suppose that $Ex(X_1)=\{a,b\}, Ex(X_1\setminus a)=\{c,b\}$, $a\not \rightarrow z$ and $az\rightarrow y$ hold in some $(X_1,\phi)$. We need to show that $cz\rightarrow y$ also holds. As discussed before, we may assume that $a\not = b\not = c$, also we may assume that $z\not = a$ because of $a\not \rightarrow z$, and that  $z\not = b$, otherwise, the conclusion follows.

Build a sequence of ordered pairs $S$: $(a,b_0=b),(a,b_1),\dots, (a,b_n=a)$, where $\{a,b_{i+1}\}=Ex(X_1\setminus\{b_0,\dots, b_i\})$. We claim that $z=b_i$ for some $i\leq n$. Indeed, if not, then $a\rightarrow z$, which contradicts the assumption. Also, $z,y \in X_1\setminus\{b_0,\dots, b_{i-1}\}$, and since $az\rightarrow y$, $y
\not = b_i$, for all $i$.

 \begin{figure}[H]
 \includegraphics[width=0.5\textwidth]{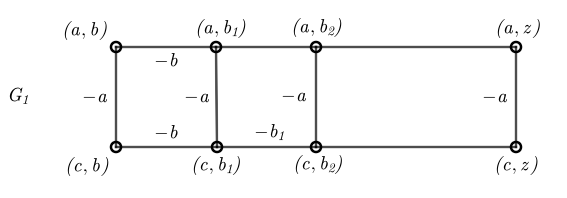}
 %{ExR.jpg}
 \caption{}
 \label{fig:ExR}
 \end{figure}
Now consider pairs of extreme points for $X_1\setminus\{a\}$, $X_1\setminus\{a,b_0\}$, $\dots X_1\setminus\{a,b_0, \dots b_{i-1}\}$.

The left-most vertical edge on the Figure \ref{fig:ExR} corresponds to $Ex(X_1)=\{a,b\}$, $Ex(X_1\setminus a)=\{c,b\}$, and the vertices along the upper edge are marked by pairs of sequence $S$. Since (Sq) holds, the second vertex along the lower edge should be marked $(c,b_1)$. If it happens that $b_1=c$, then $\{c\}=Ex(X_1\setminus\{a,b\})$, therefore, $c\rightarrow y$, hence, also $cz \rightarrow y$. Otherwise, $c\not = b_1$, and we can apply (Sq) to the second square  on the picture and conclude that $Ex(X_1\setminus\{a,b, b_1\})=\{c,b_2\})$. Proceed with this argument and either meet pair $(c,c)$ along the way on the lower edge of the picture, thus, $c\rightarrow y$, or obtain $Ex(X_1\setminus\{a,b_0, b_1, \dots b_{i_1}\})=\{c,z\}$. Since $y \in \phi(a,z)$ and $y\not = a$ and $a$ is an extreme point, then $y \in \phi(c,z)$ and $cz\rightarrow y$, as needed.  
%This follows from Theorem \ref{segments}. 

%Indeed, consider representation $(X,\phi)=(X,\leq_L) + (X,\leq_R)$. Without loss of generality, one may assume that $a$ is $\leq_L$-maximal and $b$ is $\leq_R$-maximal.
%Then $a\not \rightarrow z$ means that $a\leq_R z \leq_R b$, and $a\rightarrow y$ means $y \leq_R a$. Also, $az\rightarrow y$ really absorbs case $a\rightarrow y$ and implies that  $y\leq_R z$.

%Now $Ex(X'\setminus \{a\})=\{c,b\}$ means that $c$ is $\leq_L$-maximal in representation of $(X'\setminus \{a\})$ obtained from representation of $(X,\phi)$. Therefore, $cz\rightarrow y$ holds in $(X'\setminus \{a\})$, thus also in $(X,\phi)$.
\end{proof}

%This property says that there is a consistency of the replacement of an extreme point, when one moves from $X$ with two extreme points to $X\setminus Ex(X)$: extreme point $a$ is replaced by $c$ in both cases: when first removing $a$, or when removing $a$ after the other extreme point $b$ is removed.

\begin{lem}\label{Sq}
If $(X,\phi)$ is a convex geometry with $cdim=2$, then it satisfies (Sq).
\end{lem}
\begin{proof}
Suppose $(X_1,\phi)= (X_1, <_L) + (X_1,<_R)$ for some sub-geometry $X_1\subseteq X$, and $a$ is $<_L$-maximum and $b$ is $<_R$-maximum. Since $(c,b)$ is a pair of extreme points of $X_1\setminus \{a\}$, $c$ must be a $<_L$-maximal element after removal of $a$, so we have $c<_L a$, and no elements are in between. Since $b\not = c$, after removal of $b$ element $c$ remains to be $<_L$ maximal. As $(c,d)$ is a pair of extreme points of $X_1\setminus\{a,b\}$, we have $d<_R b$, and either no other elements are in between, or $d<_R a<_R b$. In the first case, $Ex(X_1\setminus \{b\}) = \{a,d\}$, and in the second case, $Ex(X_1\setminus \{b\}) = \{a\}$, which is needed.
\end{proof}
We can now formulate the main result.

\begin{thm}\label{main}
A convex geometry $(X,\phi)$ has $cdim=2$ iff $(X,\phi)$ satisfies properties (2Ex) and (Sq).
\end{thm}

The main effort is to show that our two properties guarantee that $cdim=2$. The Theorem will hold true due to statements of Lemmas \ref{ExR}, \ref{Sq} and \ref{main1}.

\begin{lem}\label{main1}
If convex geometry $(X,\phi)$ satisfies properties (2Ex) and (Sq) then it is represented by segments on a line.
\end{lem}
\begin{proof}
We will use induction on $|X|$. Apparently, geometry on one-element set is represented by one segment. Therefore, we assume that the geometry on $(n-1)$-element set is represented by segments as long as two properties are satisfied, and that we are give a geometry with two properties on set $X$ with $|X|=n$. 

We take $Ex(X)=\{a,b\}$ and consider $X_1=X\setminus \{a\}$. If geometry $(X,\phi)$ has properties (2Ex) and (Sq), then its projection on $X_1$ also does, therefore, $(X_1, \phi)$ is represented by segments, by inductive assumption. 

Let $(X_1,\phi)$ be represented as $(\tilde{P_k}\  \dots \ \tilde{P_2} \ \tilde{P_1}\  \nabla \ P_1\ P_2\ \dots P_k)$, where $X_1=P_1\dot\bigcup P_2\dot\bigcup \dots \dot\bigcup P_k$ is a partition, and $\tilde{P_i}$ is a permutation of $P_i$, and each projection $(P_s, \phi)$, $s\leq k$, has a unique representation. We may assume that, moreover, $(P_s, \phi) = (P_s, <_L)+(P_s,<_R)$ where the ending segments for $k$ elements are distinct for $1\leq k \leq n-1$.  

Return to the fact $Ex(X)=\{a,b\}$. If $a=b$, i.e. $a$ is a unique extreme point of $(X,\phi)$, we can place point $a$ as a maximal element in both $<_L$ and $<_R$ and obtain the representation for $(X,\phi)$. 

Therefore, we assume that $a\not = b$.
Then $b \in Ex(X_1)$, therefore, $b$ must be  a maximal element in one of two chains representing uniquely projection $(P_k, \phi)$. Without loss of generality, we may assume that the sub-chain where $b$ is a $<_R$-maximal can be switched into the right chain.

We need to show that both the left chain and the right chain have locations for $a$, which do not contradict any information about $a$ that one infers from $\phi$. Since $a \in Ex(X)$, we may assume its maximal position extending the left chain of $X_1$, while second extreme point $b$ is maximal in the right chain.

The rest of the proof is to show that an appropriate location exists for element $a$ in the right chain.

First look into the right scheme-chain of the segments: $(P_1, P_2, \dots, P_k)$.

Consider partition $X_1=Y\dot\cup Z$, where $a\rightarrow Y\setminus Z$, i.e. $Y = \phi(a)$ and $Z= X_1\setminus \phi(a)$. Let $m$ be the largest index for which $P_m \cap Y\not = \emptyset$. It follows that, for every $s > m$ (in case $m\not = k$)  and every $w \in P_s$ we have $a \not \rightarrow w$. Also, for every $t<m$ (in case $m\not = 1)$, we have $y\rightarrow u$, for each $y \in P_m$ and $u \in P_t$. Since we assumed that there exists $y \in P_m\cap Y$, we have $a\rightarrow y\rightarrow u$, for all $u \in P_t$. In other words, $a\rightarrow u$ for all elements of $P_t$ that come earlier that $P_m$ in the scheme-chain of segments $(P_1 \dots P_t \dots  P_m \dots P_s \dots  P_k )$, and $a\rightarrow z$ for all $z$ in all $P_s$ that come after $P_m$.

If $P_m\subseteq Y$, then we have a prospective location for $a$ immediately after $P_m$ and before $P_{m+1}$.

More generally, we have that $P_m= (P_m\cap Y)\cup(P_m\cap Z)$, where both sets $P_m\cap Y$ and $P_m\cap Z$ are not empty.

Consider projection $(P_m,\phi)$ and let $Ex(P_m)=\{b_1,c_1\}$. Then $Ex(P_m\cup a, \phi)=\{a,b_1\}$, or $Ex(P_m\cup a, \phi)=\{a,c_1\}$ or $Ex(P_m\cup a, \phi)=\{a\}$. The last case is already considered before, because it implies $P_m\subseteq Y$.
From two remaining cases, we may assume, without loss of generality, that $Ex(P_m\cup a, \phi)=\{a,b_1\}$. We may choose appropriate representation of $(X_1,\phi)$, where the chain from the unique representation of $(P_m,\phi)$ that has maximal element $b_1$ is located on the right, and the other chain is on the left.

Thus, the setting of the remaining problem is that we have a projection $(P_m, \phi)$ which has a unique representation: $(c_1 \dots \nabla \dots b_1)$ and $Ex(P_m\cup a)=\{a,b_1\}$. We assume that $a$ has a perspective location as a  $<_L$-maximal element, and that we need to find a location for $a$ on the right.

Element $a$ will have a proper location on the right, if 
\begin{itemize}
    \item elements from $P_m\cap Y$ form initial segment of the chain on the right;
    \item if $z\not = z' \in Z\cap P_m$ and $az \rightarrow z'$, then $z'<_R z$.
\end{itemize}

If these conditions are satisfied then $a$ can be placed as follows in the right chain: $(Y \ a \ Z)$. Moreover, the order of elements in $Z$ will be in agreement with all implications involving $a$. Observe that, for any two elements $z,z' \in Z$, we should have $Ex(a,z,z')=\{a,z\}$ or $Ex(a,z,z')=\{a,z'\}$, which implies that either $az \rightarrow z'$ or $az' \rightarrow z$ holds.

Note that both conditions above will be satisfied, if we show, for every pair $y,z \in P_m$, where $z \in Z$, that if $a\rightarrow y$ or $az\rightarrow y$, then $y <_R z$. Indeed, here $y$ plays the role of either element in $Y$ or element $z'$.

Thus, we will fix terms as following. We assume that $a \not\rightarrow z$ and [$a\rightarrow y$ or $az\rightarrow y$]. We will need to show $y <_R z$.

Let us assume that the unique representation of $(P_m, \phi)$ is\\ $(c_1\  c_2 \  \dots c_t \  (zy) \dots \nabla \dots (zy) \dots b_1)$, where $(zy)$ indicates a position on the left (or on the right), where the corresponding maximal of $y,z$ appears.

(A) First, consider the case when $y>_L z$, i.e. $y$ is maximal among $y,z$ on the left. We will show that assumption $z<_R y$ will bring to a contradiction, therefore, $y<_R z$ in this case, as desired.

Indeed, assumption $z<_R y$, together with $y<_L z$, will lead to $y\rightarrow z$ in $(P_m,\phi)$, thus, in $(X,\phi)$. If $a\rightarrow y$, we will get a contradiction with assumption $a \not\rightarrow z$. If If $az\rightarrow y$, then together with $y \rightarrow z$, the anti-exchange property of the convex geometry will lead to $a \rightarrow y,z$, again a contradiction with the assumption $a\not\rightarrow z$.

(B) Now assume that $z>_L y$. If $t=0$, i.e. $z$ is an extreme point and the maximal element in the left chain, we can apply (ExR): having $Ex(P_m\cup a)=\{a,b_1\}$ and $Ex(P_m)=\{z,b_1\}$, then [$a \not\rightarrow z$ and ($a\rightarrow y$ or $az\rightarrow y$)] implies [$z\rightarrow y$]. This would imply that $y<_Rz$, as needed.

So now assume that $t\geq 1$.  By assumption, $(P_m, \phi)$ has (ES) property of Lemma \ref{seq}: the last segments of $t$ elements in two chains should be distinct, when $1\leq t \leq n-1$. Let $p$ be the largest element in the the final $t$ elements in the right chain that is distinct from all $c_1,\dots, c_t$.

Apply Lemma \ref{seq} to $(P_m,\phi)$ to build a sequence $x_1,\dots x_n$, and then use it to produce sequence of ordered pairs $(u_1,v_1), \dots , (u_k,v_k)$ as follows: 
\begin{itemize}
    \item [(1)] $(u_1,v_1)=(c_1,b_1)$ and $(u_{i+1},v_{i+1})$ is an ordered pair of extreme points of $P_m\setminus \{x_1,\dots, x_i\}$, $i\leq k$;
    \item [(2)] $u_i\not = v_i$, for all $i\leq k$;
    \item[(3)] $v_k=p$ and ($u_k=c_i$ or $u_k=z$).
   % \item [(4)] for each $i < k$, either $u_i=u_{i+1}$, and $v_i = x_i$ so that $v_{i+1}$ is a new extreme point after removal of $v_i$, or $v_i=v_{i+1}$ and $u_{i+1}$ is a new extreme point after removal of $u_i$.
\end{itemize}

We can proceed by removing extreme points $x_1, \dots, x_n$ either on the left or on the right, so that remaining projection has (ES) property and (1) and (2) are satisfied. We only need to show that on some $k$th step in the process we will have a pair satisfying (3).

If $p=b_1$, then (2) is already satisfied. If $p <_R b_1$, then let $T>0$ be the number of elements in the right chain greater than $p$. According to the definition of element $p$, all these $T$ elements will appear among $c_1,\dots c_t$.  We will show that $b_1=x_j$ for some $j\leq k$, thus $b_1$ will be removed at step $j$ and the remaining projection will still have property (ES) described in Lemma \ref{seq}. Since $b_1$ appears among $c_1, \dots c_t$ we have $c_1 >_L \dots >_L c_s >_L b_1=c_{s+1} >_L \dots >_L c_t$. We cannot have $x_1=c_1, x_2=c_2 \dots x_s=c_s$, because removing $c_s$ we obtain   $(b_1,b_1)$, which violates (ES) property. Therefore, $b_1=x_j$ for some $j\leq s$, so that the remaining projection with end points $(u_{j+1},v_{j+1})$ still satisfies (ES), but now $p<_R v_{j+1}$ with less than $T$ elements $>_R p$.

If $p=v_{j+1}$, then we already have pair satisfying (3). Otherwise, we will have $p <_R v_{j+1}$, where $v_{j+1}$ was not among $\{x_1,\dots x_j\}\subseteq \{c_1,\dots, c_t\}$, thus, it still appears among remaining $c_1, \dots c_t$ in the left chain. So the same argument can be applied to a projection with extreme points $(u_{j+1},v_{j+1})$.

Thus, we have a sequence $S$: $(c_1,b_1), (u_2,v_2) \dots (u_k,v_k)=(c_i,p)$ or $(u_k,v_k)=(z,p)$, where $x_1\in \{c_1,b_1\}$, and $x_s\in \{u_s,v_s\}$, for $1 < s < k$.

Recall that $\{c_1,b_1\}=Ex(P_m)$ and $\{a,b_1\}=Ex(P_m\cup \{a\})$. We now want to apply the same sequence $x_1, \dots, x_k$ of removals to $(P_m\cup \{a\})$ and observe the ordered pairs of extreme points in the process.

If $x_1=c_1$, then $(u_2,v_2)=(c_2,b_1)$. Since $c_1$ is not an extreme point of $(P_m\cup \{a\})$, we have $(a,b_1)=Ex((P_m\cup \{a\}\setminus\{x_1\})$, thus, the pair of extreme points does not change when $x_1=c_1$. On the other hand, if  $x_1=b_1$, and say $(u_2,v_2)=(c_1,b_2)$, then by (Sq) property we will have $Ex((P_m\cup \{a\}\setminus\{b_1\})=\{a, b_2\}$ or $\{a\}$. Therefore, removal of $x_1$ in $(P_m\cup \{a\})$ will bring to pair $(a,b_2)$ in second case, which is a change from $b_1$ to $b_2$ in second component, like in case of sequence $S$ after the first step of removal of $x_1$. Alternately, we can get $(a,a)$ which implies that $a \rightarrow P_m\setminus\{b_1\}$. In this case, all consecutive steps of removals of $x_i$ will not change pair $(a,a)$.
 
Applying the same argument to the step when $x_2$ is removed, we observe that the ordered pair of extreme points of $(P_m\cup \{a\}\setminus\{x_1,x_2\})$ does not change, if $x_2=u_2$, or changes to $(a,v_3)$ or $(a,a)$, if $x_2=v_2$.

To illustrate the process, consider example on Figure \ref{fig:seq1}. The picture gives a partial representation of two chains of $(P_m,\phi)$, with parameters $t=4$ and $T=3$.

\begin{figure}[H]
 \includegraphics[width=0.5\textwidth]{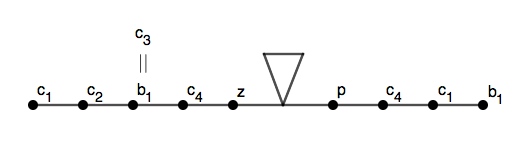}
 %{Example.jpg}
 \caption{}
 \label{fig:seq1}
 \end{figure}

\begin{figure}[H]
 \includegraphics[width=0.8\textwidth]{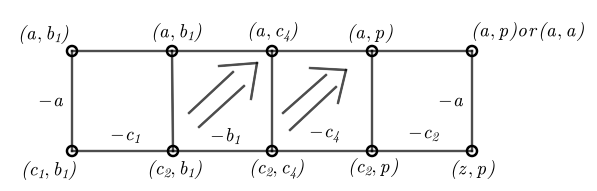}
 %{Example.jpg}
 \caption{}
 \label{fig:seq2}
 \end{figure}
 
On Figure \ref{fig:seq2}, the sequence $(u_1,v_1), \dots (u_5,v_5)$ is given along the lower edge with the sequence of removed elements  $x_1=c_1,x_2=b_1,x_3=c_4, x_4=c_2$. Note that pair $(c_2,p)$ is obtained after three steps. The upper edge corresponds to the the same sequence of removed elements in geometry $Ex(P_m\cup \{a\})$. It shows that the pair does not change when $x_i$ is the left-end point in the corresponding pair on the lower edge, and (Sq) is applied with indicating arrow within the square, when $x_i$ is the right-end point of a pair on the lower edge.

%that has distinct extreme points $c_1,b_1$. One can remove extreme point on each step so that the remaining geometry will have exactly two extreme points. Thus, after at most $t$ steps one gets a sub-geometry with extreme points: $c_i$ or $z$ on the left and $p$ on the right. By the (Sq) property, every removal of an element that is extreme point on the right is coordinated by replacement of extreme point in sub-geometry that has $a$ (need a picture here), and when an extreme point is removed on the left, it does not effect the extreme points of sub-geometry with $a$. Therefore, after at most $t$ applications of (Sq) we will get a sub-geometry $P_m^*$: $Ex(P_m^*\cup a)=\{a,p\}$ and $Ex(P_m^*)=\{c_i,p\}$ or $Ex(P_m^*)=\{z,p\}$.

Denote $P_m^*=P_m\cup\{a\}\setminus\{x_1,\dots, x_k\}$.

If $p=y$, then we have $ay\rightarrow z$,which is together with assumption $a\rightarrow y$ or $az\rightarrow y$, as well as the the anti-exchange property, will imply $a\rightarrow z$, a contradiction. If $p=z$, then we have $y\leq_R z$, which is needed.

Suppose we have $Ex(P_m^*)=\{a,p\}$ and either $Ex(P_m^*)=\{c_i,p\}$ or $Ex(P_m^*)=\{z,p\}$. In the first case, consider $P_m^{**}=P_m^*\setminus \{c_i,\dots, c_t\}$. Then $Ex(P_m^{**})=\{a,p\}$, because all removed points are not extreme. On the other hand, $Ex(P_m^{**}\setminus \{a\})=\{z,p\}$, by assumption on points $c_1,\dots, c_t$ and $p$.

Then we can apply property (ExR): given $Ex(P_m^{**})=\{a,p\}$ and $Ex(P_m^{**}\setminus \{a\})=\{z,p\}$, the assumption of [$a\rightarrow \setminus z$ and ($a\rightarrow y$ or $az\rightarrow y$)] implies [$z\rightarrow y$]. The last implication $z\rightarrow y$ yields that $y<_R z$ on the right, which is needed.

Finally, if $Ex(P_m^*)=\{a\}$, then $a \rightarrow z$, a contradiction with $a\not \rightarrow z$.

This finishes the proof.
\end{proof}

At the end we want to discuss the question of the complexity of the problem to recognize whether a convex geometry given by its operator, say, by implicational basis $S$, has $cdim=2$. We will assume that the geometry is defined on $X$ with $|X|=n$ by the basis of $m$ implications, and the size of $S = \{A_i\rightarrow B_i\}$ is computed as $s(S)=\Sigma |A_i| + \Sigma |B_i|$.

\begin{lem}\label{triple}
For any closure system $(X,\phi)$ property (2Ex) holds iff for any $a,b,c \in X$ one of implications $ab\rightarrow c$, $ac\rightarrow b$, $bc\rightarrow a$ holds.  
\end{lem}

\begin{proof}
If (2Ex) holds and $a,b,c \in X$, then set $X'=\{a,b,c\}$ should have at most two extreme points, therefore, one of points is not extreme. If it is, say, $c$, then $ab\rightarrow c$ holds.

Now if (2Ex) fails, then some $X'\subseteq X$ has at least three extreme points $a,b,c$. Then none of $ab\rightarrow c$, $ac\rightarrow b$, $bc\rightarrow a$ would hold. Thus, the other property of Lemma fails as well.
\end{proof}

\begin{cor}
Let $(X,\phi)$ be a convex geometry given by implicational basis $S$, $|X|=n$, $|S|=m$ and $s(S)=k$.
The number of steps required to verify that $cdim=2$ is $O((k+m)n^3)$. 
\end{cor}
\begin{proof}
According to Lemma \ref{triple}, we need to verify, for each triple of elements from $X$, whether one of elements is in a closure of two others. It takes linear time on the size of the basis to verify that $a \in \phi(b,c)$. Thus, it takes $O(kn^3)$ time to verify (2Ex).

For any $a,b \in X$, $Ex(\phi(a,b))$ is either $a$, or $b$ or $a,b$, so the property (Sq) has to be checked for pairs $(a,b)$ that are extreme points of some closed set. To find extreme point $c$ of $\phi(a,b)\setminus \{a\}$ takes $O(mn)$ steps, by verifying that $c$ does not appear as a consequent of any implication that does not involve $a$, see Lemma 14. Thus, it takes $O(n^2\cdot mn)=O(mn^3)$ to check (Sq).  
\end{proof}

{\it Acknowledgements.} We are grateful to Hofstra University that provided funds for both authors to travel to the conference Algebras and Lattices in Hawai'i-2018, where the results of this paper were presented. We want to thank Madina Bolat (University of Illinois in Urbana-Champaign), for her valuable comments and the help in producing the pictures. 

\end{document}